\documentclass{amsart}
\usepackage{amsmath, amssymb}
\usepackage{comment}
\newcommand{\eeq}{\end{equation}}

\usepackage{xcolor}
\usepackage{caption}
\usepackage{subcaption}
\usepackage{smartdiagram}
\smartdiagramset{module minimum width=0, module minimum height=0, text width=1em, uniform color list=gray!20 for 10 items, arrow color=gray, uniform arrow color=true}
\usepackage{graphicx}

\usepackage{etoolbox}
\patchcmd{\thebibliography}{\chapter*}{\section*}{}{}
\usepackage{ltablex}

\usepackage[nospace,noadjust]{cite}
\usepackage{calc}
\usepackage{geometry}
\usepackage{cite}
\usepackage{tikz-cd}

\usepackage{tikz}
\usetikzlibrary{matrix}

\usepackage{indentfirst}
\usepackage{longtable}

\allowdisplaybreaks

\numberwithin{equation}{section} 

\newcommand{\lra}{\longrightarrow}

\newcommand{\bP}{\mathbb{P}}
\newcommand{\bM}{\mathbb{M}}
\newcommand{\bZ} {\mathbb{Z}}
\newcommand{\bC}{\mathbb{C}}
\newcommand{\bQ}{\mathbb{Q}}
\newcommand{\bQb}{\overline \bQ}

\newcommand{\cS}{\mathcal{S}}
\newcommand{\CL}{\mathcal{L}}
\newcommand{\cP}{\mathcal{P}}
\newcommand{\cO}{\mathcal{O}}
\newcommand{\cJ}{\mathcal{J}}
\newcommand{\cU}{\mathcal{U}}

\DeclareMathOperator{\GL}{GL}
\DeclareMathOperator{\tors}{tors}
\DeclareMathOperator{\Aff}{Aff}
\DeclareMathOperator{\Bir}{Bir}
\DeclareMathOperator{\Char}{char}
\DeclareMathOperator{\Prep}{Prep}
\DeclareMathOperator{\End}{End}

\newcommand{\Kbar}{ \overline{K}}

\newcommand{\R}{\mathbb R}
\newcommand{\bR}{\mathbb R}
\newcommand{\Q}{\mathbb{Q}}
\newcommand{\Qbar}{\overline{\Q}}
\newcommand{\F}{\mathbb{F}}
\newcommand{\fh}{\mathfrak{h}}

\newtheorem{thm}{Theorem}[section]

\newtheorem{prop}[thm]{Proposition}
\newtheorem{lem}[thm]{Lemma}
\newtheorem{cor}[thm]{Corollary}
\newtheorem{quest}[thm]{Question}

\theoremstyle{definition}

\newtheorem{remark}[thm]{Remark}
\newtheorem{example}[thm]{Example}

\title{A Tits alternative for rational functions}

\author{Jason P. Bell} \thanks{The first-named author thanks NSERC for
  its generous support.}  \address{ Department of Pure Mathematics,
  University of Waterloo, 200 University Ave. W., Waterloo, ON, N2L
  3G1, CANADA}\email{jpbell@uwaterloo.ca}

\author{Keping Huang}
\address{Department of Mathematics, Michigan State University, East
  Lansing, MI 48824, USA} \email{keping.huang@rochester.edu}

\author{Wayne Peng} \address{Department of Mathematics, University of
  Rochester, Rochester, NY 14627 USA}
\email{junwen.wayne.peng@gmail.com}

\author{Thomas J. Tucker}\thanks{The fourth-named author thanks the
  Mathematical Sciences Research Institute for
  its generous support.} 
\address{Department of Mathematics, University of Rochester,
  Rochester, NY 14627 USA} \email{thomas.tucker@rochester.edu}

\date{October 14, 2020}

\begin{document}

\begin{abstract} We prove an analog of the Tits alternative for
  rational functions.  In particular, we show that if $S$ is a
  finitely generated semigroup of rational functions over $\bC$, then
  either $S$ has polynomially bounded growth or $S$ contains a
  nonabelian free semigroup.  We also show if $f$ and $g$ are
  polarizable maps over any field of any characteristic and
  $\Prep(f) \not= \Prep(g)$, then there is a positive integer $j$ such that
  $\langle f^j, g^j \rangle$ is a free semigroup on two generators.  In the special case of
  polynomials, we are able to prove slightly stronger results.
\end{abstract}

\maketitle

%
%
%
%
\section{Introduction}
The Tits alternative \cite{Tits} is a celebrated result in the theory
of linear groups.  It says that a finitely generated linear group
contains either a solvable subgroup of finite index or a 
nonabelian free group.  In general, a group $G$ is said to satisfy the
\emph{Tits alternative} if each of its finitely generated subgroups
contains either a solvable subgroup of finite index or a nonabelian
free group.  Many classes of groups have now been shown to satisfy
the Tits alternative \cite{BFH, KOZ, Iva, Lam, McC}.

When one instead considers the structure of linear groups as
semigroups, an even stronger dichotomy is obtained.  A result of
Longobardi, Maj, and Rhemtulla \cite{LMR} (see also Milnor \cite{Mil}
and Wolf \cite{Wol}) combined with the Tits alternative implies that a
finitely generated linear group is either virtually nilpotent or
contains a nonabelian free semigroup.  Okni\'nski and Salwa \cite{Ok1}
later showed if $S$ is a finitely generated cancellative linear
semigroup, then either $S$ contains a nonabelian free semigroup or the
group generated by $S$ is virtually nilpotent.  Results from the
theory of growth of groups then give that the growth of a finitely
generated cancellative linear semigroup is either exponential or is
polynomially bounded.  The cancellativity condition here is crucial as
Okn\'inski \cite{Ok2} has also produced finitely generated
non-cancellative linear semigroups of intermediate growth (see also
\cite{Ok3}).

We note that a semigroup $S$ contains a nonabelian free semigroup if
and only if it contains a free semigroup on two generators.  As with
groups, it is not difficult to see that a free semigroup on two
generators must contain a free semigroup on $n$ generators for any
positive integer $n$.



We prove the following variant of the Tits alternative for semigroups
of rational functions over the complex numbers.  

\begin{thm}\label{rational}
  Let $\cS$ be a finitely generated semigroup
  of rational functions in $\bC(x)$.  Then either $\cS$ has polynomially
  bounded growth or $\cS$ contains a nonabelian free semigroup. 
\end{thm}

We say that a rational function of degree greater than 1 is {\em
  non-special} if it is not conjugate to a monomial, a Chebychev
polynomial, or a Latt\`es map.  When $\cS$ contains a non-special
rational function of degree greater than one, we obtain a stronger
dichotomy.

\begin{thm}\label{linear}
Let $\cS$ be a finitely generated semigroup of rational functions in
$\bC(x)$ such that some element of $\cS$ is a non-special rational
function of degree greater than 1.  Then either $\cS$ has linear
growth or $\cS$ contains a nonabelian free semigroup.  
\end{thm}

We derive Theorem \ref{rational} from a result relating common
preperiodic points of rational functions with free subsemigroups.  The
techniques used for this result work in the setting of
morphisms of projective varieties that are polarized by the same ample
line bundle.  For $V$ a projective variety, a morphism $f: V \lra V$
is said to be {\em polarized} by the ample line bundle $\CL$ if there
is a $d > 1$ such that $f^*\CL \cong \CL^{\bigotimes d}$.  The notion
of polarization is due to Zhang \cite{ZhangSmall}.  Any morphism of
degree greater than 1 on projective space $\bP^n$ is polarized by
$\cO_{\bP^n}(1)$; it is also true that any polarized morphism on a
variety $V$ comes from restricting a morphism of projective space to
$V$ for some embedding of $V$ into projective space (see \cite{Fak,
  Lucien}).  Polarized morphisms give rise to canonical height
functions with good properties (see \cite{CS93} and Section
\ref{height-section}).  In the theorem below and throughout this
paper, we let $\langle f, g \rangle$ denote the semigroup generated
under composition by $f$ and $g$ whenever $f$ and $g$ are two
functions from a set to itself.  We also let $\Prep(f)$ and $\Prep(g)$
denote the set of preperiodic points of $f$ and $g$, respectively.
 
\begin{thm}\label{common}
  Let $V$ be a projective variety, and let $f, g: V \lra V$ be
  polarized by the same line bundle $\CL$. If
  $\Prep(f) \not= \Prep(g)$, then there is a positive integer $j$ such that $\langle
  f^j, g^j \rangle$ is a free semigroup on two generators.  
\end{thm}

In the case of polynomials, we can derive stronger results in some
cases.

\begin{thm}\label{poly1}
  Let $\mathcal{S}$ be a finitely generated semigroup of polynomials
  in $K[x]$ for $K$ a field such that $\Char K$ does not divide the
  degrees of any element of $\cS$.   Then either $\cS$ has polynomially
  bounded growth or $\cS$ contains a nonabelian free semigroup. 
\end{thm}

\begin{thm}\label{poly2}
  Let $K$ be a field and let $f, g \in K[x]$ have degree greater than
  1.  Suppose that $\Char K$ does not divide $\deg f$ or $\deg g$.  If $\Prep(f) \not= \Prep(g)$, then $\langle f, g \rangle$ is a free semigroup on two generators.
\end{thm}

When $f \in \bC[x]$ has degree greater than 1, its Julia set is
determined by $\Prep(f)$.  Thus, Theorem \ref{poly2}
implies that if $f, g \in \bC[x]$ have degree greater than one and do
not share the same Julia set, then $\langle f, g \rangle$ is a free semigroup on two generators.

We are also able to prove the following for abelian varieties of any
dimension in any characteristic.

\begin{thm}\label{av-thm}
  Let $A$ be an abelian variety.  Let $\cS$ be a finitely generated
  semigroup of finite morphisms from $A$ to itself.  Then either $\cS$ has
  polynomially bounded growth or $\cS$ contains a nonabelian free semigroup.  
\end{thm}

Since every irreducible curve $C$ with an infinite semigroup of
nonconstant maps $f: C \lra C$ has genus 0 or 1, this means that any
semigroup of morphisms from a curve to itself either has polynomial
growth or contains a nonabelian free semigroup (see Corollary
\ref{all-curves}).

Ritt \cite{R1} studied the semigroup of polynomials under composition
and gave necessary and sufficient conditions for two polynomials to
commute under composition and determined relations for the semigroup
of polynomials under composition.  It is very possible that in this
case, some of the results here can be obtained using Ritt's work,
although there do appear to be some additional subtleties involved.
We also point out that the Tits alternative has been considered for
automorphism groups of algebraic varieties, with a complete result for
projective varieties in characteristic 0 (see \cite{T-C} for a survey) as
well as some results in characteristic $p$ (see \cite{Hu}).  The Tits
alternative has also been proved for the Cremona group
$\Bir(\bP^2)$ in all characteristics (see \cite{Cantat}). 

Pakovich \cite{Pak} has proved if $\cS$ is a semigroup of non-special
polynomials over the complex numbers, then either $\cS$ contains a
nonabelian free semigroup or $\cS$ is amenable.  He also
extends this to certain classes of rational functions.  We note that
one can use F{\o}lner conditions along with Theorem \ref{linear} to
prove that if $\cS$ is any finitely generated semigroup of non-special
rational functions in characteristic 0, then either $\cS$ contains a
nonabelian free semigroup or $\cS$ is left amenable.  On
the other hand, it is not clear how to adapt our techniques
to semigroups that are not finitely generated, though it may be
possible to do so in the non-special case.  Hindes \cite{Hindes} has
proved that certain conditions on a semigroup of rational functions
over $\bQb$ guarantee that the semigroup is free; the conditions are
much more restrictive than those of Theorem \ref{common} but have the
advantage of ensuring that certain semigroups are free (and don't
merely contain a nonabelian free semigroup).

An outline of the paper is as follows.  We begin with some
preliminaries on semigroups and growth in semigroups in Section
\ref{prelim}.  In Section \ref{common-sec}, we prove Theorem
\ref{common}.  We begin the proof in Section \ref{ping} with a proof
of Proposition \ref{free}, a variant on the ping-pong lemma (see
\cite{Tits}), that can be applied to a wide variety of functions.
Then we introduce canonical height functions, both Weil and Moriwaki,
which allow us to use Proposition \ref{free} to derive Theorem
\ref{common}.  We close the section with Examples \ref{AV} and
\ref{p-counter}, which show that a converse to Theorem \ref{common} is
not possible in higher dimensions in characteristic 0 or even in
dimension 1 in characteristic $p$.  In Section \ref{linear-sec}, we
prove Theorems \ref{linear} and \ref{rational}.  We do so by proving
Propositions \ref{g1} and \ref{g2}, which may be thought of as
converses to Theorem \ref{common} in the special case of rational
functions in characteristic 0; Theorems \ref{linear} and
\ref{rational} follow immediately from combining Propositions \ref{g1}
and \ref{g2} with Theorem \ref{common}.  Proposition \ref{g1} treats the
  case of non-special rational functions; the proof uses a result due
  to Ye (see \cite[Theorem 1.5]{Ye}) along with an argument about the
  action of a semigroup of morphisms on a point with a finite orbit
  under that semigroup.  The proof of Proposition \ref{g2} is then a
  case-by-case analysis of the different sorts of special rational
  functions.  Section \ref{poly-sec} contains the proofs of Theorems
  \ref{poly1} and \ref{poly2}, which follow quickly from some lemmas
  of Jiang and Zieve \cite{JZ} about B\"ottcher coordinates and formal
  power series.  We then prove Theorem \ref{av-thm} in Section
  \ref{abelian-v} using standard results on abelian varieties along
  with a theorem from \cite{Ok1}.  We conclude the paper in Section
  \ref{further-sec} with some questions.

\vskip2mm
\noindent {\em Acknowledgments.}  We would like to thank Alex Carney,
Dragos Ghioca, Wade Hindes, Liang-Chung Hsia, Patrick Ingram, Shu
Kawaguchi, Fedor Pakovich, Juan Rivera-Letelier, Sema Salur, Hexi Ye,
Shouwu Zhang, and Michael Zieve for many helpful conversations.  We
would also like to thank the Simons Foundation and the University of
Paderborn for hosting conferences at which some of the problems here
were discussed.

\section{Preliminaries}\label{prelim}

We give a brief overview of the basics of semigroups.  Let $\cS$ be a finitely generated semigroup and let $S$ be a finite set of generators for $\cS$.  Then we can form the \emph{growth function} of $\cS$ with respect to the generating function $S$ as follows.  We define $d_S(n)=|S^{\le n}|$, where $S^{\le n}$ is the set of elements of $\cS$ that can be expressed as a product of elements of $S$ of length at most $n$.  The function $d_S(n)$ is weakly increasing as a function of $n$ and while this function depends upon our choice of generating set, we observe that if $T$ is another generating set for $\cS$ then there exists a positive integer $a$ such that $T\subseteq S^{\le a}$ and $S\subseteq T^{\le a}$ and so we have the inequalities
$$d_S(n)\le d_T(an)\qquad {\rm and}\qquad d_T(n)\le d_S(an).$$ Thus if we declare that two weakly increasing functions $f,g:\mathbb{N}\to \mathbb{N}$ are \emph{asymptotically equivalent} if there is a positive integer $C$ such that $f(n)\le g(Cn)$ and $g(n)\le f(Cn)$ then the growth function is independent of our choice of generating set up to this asymptotic equivalence.  
Given a finitely generated semigroup $\cS$ with finite generating set
$S$, we say that $\cS$ has \emph{polynomially bounded growth} if
$d_S(n)={\rm O}(n^{\kappa})$ for some positive constant $\kappa$; we
say that $\cS$ has \emph{linear growth} if there are positive
constants $C_1, C_2$ such that $C_1n \leq d_S(n) \leq C_2 n$ for all
$n$; and we say that $\cS$ has \emph{exponential growth} if there is a
positive constant $C>1$ such that $d_S(n)>C^n$ for all $n$
sufficiently large.  It is not difficult to check that the properties
of having linear growth, polynomially bounded growth, and exponential
growth are all preserved under asymptotic equivalence and so we can
speak unambiguously of $\cS$ having these properties without making
reference to a generating set.

A semigroup $\cS$ is \emph{left cancellative} if whenever $ax=ay$ with
$a,x,y\in \cS$ we have $x=y$; right cancellativity is defined
analogously. A \emph{cancellative} semigroup is one that is both left
and right cancellative.  Note that if $\cS$ is a semigroup of
surjective maps, then $\cS$ is right cancellative since $xa = ya$
implies that $x=y$ whenever $a$ is surjective.  Hence, in particular,
semigroups of nonconstant rational functions are right cancellative;
on the other hand $X^2 \circ (-X) = X^2 \circ (X)$, so semigroups of
rational functions are not always left cancellative.  

We will introduce the theory of height functions in the next section.
It seems more natural to us to present them in the context of the proof of
Theorem \ref{common} than to do so in advance.

\section{Proof of Theorem \ref{common}}\label{common-sec}

\subsection{A variant of the ping-pong lemma} \label{ping}

In our work here, the functions $\tau$, $\tau_f$, and $\tau_g$ will be
some sort of real-valued height functions (either Weil or Moriwaki).
The arguments in this section work in a more general setting, and we
state Proposition \ref{free} accordingly.  

Let $\cU$ be a set, let $f,g: \cU \lra \cU$ be surjective maps, and
let $\tau: \cU \lra \R$ be any function that is not bounded in
absolute value.
Suppose that there are positive real numbers $d_1, d_2 > 1$, and a
real number $C$ such that
\begin{equation} \label{bound}
 \begin{split}
   |\tau(f(z)) - d_1 \tau(z)| < C, \\
   |\tau(g(z)) - d_2 \tau(z)| < C.
 \end{split}
\end{equation}
We say that $d_1$ is the {\em degree} of $f$ and $d_2$ is the {\em
  degree} of $g$ and write $\deg f = d_1$, $\deg g = d_2$.

Recall that one can use \eqref{bound} to construct canonical functions
  as follows
  \begin{equation}\label{can}
 \begin{split}
  \tau_f(z) = \lim_{n \to \infty} \frac{\tau(f^n(z))}{d_1^n} =
  \tau(z) + \sum_{i=0}^\infty \frac{\tau(f^{i+1}(z)) -  d_1 \tau(f^i(z))}{d_1^{i+1}},
  \\
   \tau_g(z) = \lim_{n \to \infty} \frac{\tau(g^n(z))}{d_2^n} =
  \tau(z) + \sum_{i=0}^\infty \frac{\tau(g^{i+1}(z)) -  d_2 \tau(g^i(z))}{d_2^{i+1}}.
\end{split}
\end{equation}
Note that
\[ \left| \sum_{i=0}^\infty \frac{\tau(f^{i+1}(z)) -
    d_1\tau(f^i(z))}{d_1^{i+1}} \right| < C \sum_{i=1}^\infty \frac{1}{d_1^i} \] and
\[ \left|  \sum_{i=0}^\infty \frac{\tau(g^{i+1}(z)) -
    d_2 \tau(g^i(z))}{d_2^{i+1}} \right|  < C \sum_{i=1}^\infty
  \frac{1}{d_2^i}.  \]
The telescoping sum argument above is due to Tate and was used by Call
and Silverman \cite{CS93} in their construction of canonical heights.
Kawaguchi \cite{Kawa1, Kawa2} has further developed the theory of canonical
heights in the context of semigroups.

Let $d = \min(d_1, d_2)$, and let
\[ C ' = \sum_{i=1}^\infty \frac{1}{d^i}. \]

Using the same
telescoping series argument, we see that for any $n$, we have
\begin{equation}\label{close}
  \begin{split}
  \left| \frac{\tau(f^n(z))}{d_1^n} - \tau_f(z) \right| < \frac{C'}{d_1^n}, \\
  \left| \frac{\tau(g^n(z))}{d_2^n} - \tau_g(z) \right| < \frac{C'}{d_2^n}.
\end{split}
\end{equation}

We will prove the following variant of Tits' ping-pong lemma (see \cite{Tits}).  

\begin{prop}\label{free}
  Let $\cU$ be a set, let $\tau: \cU \lra \bR$ be a function that is
  unbounded in absolute value and let $f,g: \cU \lra \cU$ be
  surjective maps that satisfy
  \eqref{bound}.  Let $\tau_f$ and $\tau_g$ be as defined in
  \eqref{can}.  Suppose that there is some $z \in \cU$ such that
  $\tau_f(z) \not= \tau_g(z)$.  Then there is a positive integer $j$ such that
\begin{itemize}
\item[(i)] we have $af^j \not= b
g^j$ for all $a,b \in \langle f, g \rangle$; and
\item[(ii)] the semigroup $\langle f^j, g^j \rangle$ is a free semigroup on two generators.
\end{itemize}
\end{prop}

We begin with one more definition.  Let $w= \varphi_m \cdots
\varphi_1$, where each $\varphi_j$ equals $f$ or $g$.  We define the
degree of $w$ as 
\begin{equation}\label{deg}
\deg w = \prod_{j=1}^m \deg \varphi_j.
 \end{equation}
 Since $|\tau(w(z)) - \deg w \tau(z)|$ is bounded for all $z$ (by
 \eqref{bound}) and $\tau$ is unbounded, we see that the definition in
 \eqref{deg} is independent of the word representing $w$.

\begin{lem}
  Let $w = \varphi_m \cdots \varphi_1$ where $\varphi_i$ is equal
  to $f$ or $g$ for each $i$.  Let $s_i = \varphi_i \cdots \varphi_1$ (for $i \leq
  m$).  
    With notation as above, we have
   \begin{equation}\label{w}
     \left| \frac{\tau(w(z))}{\deg w}  - \frac{\tau(s_j(z))}{\deg s_j} \right| <
     \frac{C'}{d^j}  
   \end{equation}
   for $j = 1, \dots, m$.  
  \end{lem}

  \begin{proof}
   Let $e_\ell = \deg \varphi_\ell$ for each $\ell =1, \dots, m$.  Then $\deg s_j
   =\prod_{\ell = 1}^{j} e_\ell$.  Thus, as in \eqref{bound}, we have a
   telescoping series
   \begin{equation}\label{tele}  \frac{\tau(w(z))}{\deg w}  - \frac{\tau(s_i(z))}{\deg s_i} =
     \sum_{j=i}^{m-1} \frac{\tau(s_{j+1}(z)) - e_{j+1}
       \tau(s_j(z))}{\prod_{\ell = 1}^{j+1} e_\ell}.
   \end{equation}
   Now
   \[ |\tau(s_{j+1}(z)) - e_{j+1} \tau(s_j(z))| = |\tau(\varphi_{j+1}(s_j(z)) -
     (\deg \varphi_{j+1}) \tau(s_j(z)) | < C \]
   for all $j$ by \eqref{bound}.  Thus
   \[   \left| \sum_{j=i}^{m-1} \frac{\tau(s_{j+1}(z)) - e_{j+1}
       \tau(s_j(z))}{\prod_{\ell = 1}^{j+1} e_\ell} \right| \leq
   \frac{1}{d^j} \sum_{i=1}^\infty \frac{C}{d^i} \leq \frac{C'}{d^j}. \]
This completes the proof, by \eqref{tele}.  
 \end{proof}

 Now, we are ready to prove Proposition \ref{free}.  
 \begin{proof}[Proof of Proposition \ref{free}]
We choose $\beta$ so that $\tau_f(\beta) \not=
   \tau_g(\beta)$.    Let $\epsilon = |\tau_f(\beta) - \tau_g(\beta)|$.  Choose
   $j$ so that $C'/d^j < \epsilon/4$, where $d = \min(d_1, d_2)$ as
   above.  Let $a f^j$ and $b g^j$ be words in $f$ and $g$ such
   that $\deg a f^j = \deg b g^j$.
   Then, by \eqref{close} and \eqref{w}, we we have
   \begin{equation*}
     \begin{split}
     \left| \tau_f(\beta) - \frac{\tau(a f^j(\beta))}{\deg a f^j} \right| <
     \epsilon /2, \\
     \left |\tau_g(\beta) - \frac{\tau(b g^j(\beta))}{\deg b g^j} \right| <
     \epsilon /2.
\end{split}
   \end{equation*} 
     Thus,
     \[ \frac{\tau(a f^j(\beta))}{\deg a f^j} \not= \frac{\tau(b
         g^j(\beta))}{\deg b g^j}.\]
     Since  $\deg a f^j = \deg b g^j$, this means that
     \begin{equation}\label{not} 
       a f^j(\beta) \not= b g^j(\beta).
     \end{equation}

   Now, let $u = \varphi_m \cdots \varphi_1$ and $w = \theta_n \cdots
   \theta_1$, where each $\varphi_i$ and $\theta_k$ is equal to $f^j$
   or $g^j$.  Suppose that $u = w$.  We will show by induction on
   $\max(m,n)$ that $m =n $ and $\theta_i = \varphi_i$ for $i=1,
   \dots, m$.  If $m = n =1$, then we must have $\theta_1 = \varphi_1$
   since $f^j \not= g^j$, because $\tau_f \not= \tau_g$.  For the inductive step, it will
   suffice to show that $\varphi_1 = \theta_1$ since we may then
   cancel (as $f^j$
   and $g^j$ are surjective) to obtain $\varphi_m \cdots \varphi_2
   = \theta_n \cdots \theta_2$ and apply the inductive hypothesis.  We
   argue by contradiction.  
   If $\varphi_1 \not= \theta_1$, then we may assume without
   loss of generality that $\varphi_1 = f^j$ and $\theta_1 =
   g^j$.  But then since $\deg u = \deg w$ (because $u = w$), we
   must have 
 \[ \varphi_m \cdots \varphi_1(\beta) \not= \theta_n \cdots
   \theta_1 (\beta), \]
 by \eqref{not}, a contradiction, so  $\theta_1 = \varphi_1$, and our proof is
complete.

 \end{proof}

 \subsection{Height functions}\label{height-section}

 We will prove Theorem \ref{common} by letting $\tau$ be a height
 function, either a Weil height $h$ or a Moriwaki height $\fh$, and
 using Proposition \ref{free}, using the fact that the canonical
 heights attached to these will be zero at exactly the points that are
 preperiodic.  There may be other sorts of functions where Proposition
 \ref{free} may be used though.  For example, if we let $\tau: \bC
 \lra \bR$ be defined by $\tau(z) = \log (\max |z|, 0)$, and $f,g \in
 \bC[x]$ are polynomials of degree greater than 1, we see that $\tau_f$ and
 $\tau_g$ vanish precisely on the filled Julia sets of $f$ and $g$
 respectively. Since the Julia set is simply the boundary of the
 filled Julia set, Proposition \ref{free} thus implies that if the
 Julia sets of $f$ and $g$ are not equal, then there is a $j$
 such that $\langle f^j, g^j \rangle$ is a free semigroup on two generators.

 For a more general exposition of the Weil height functions, see \cite{DioGeo}
 and \cite{BG}.  The Moriwaki height functions we use were introduced
 in \cite{Mori1, Mori2}.  

 Let $V$ be a projective variety and let $f, g: V \lra V$.  Since $V$
 is finitely presented, there is a finitely generated field $K$ such
 that $f$, $g$, and $V$ are all defined over $K$.  If $K$ is not
 finite, then there is a set $\bM_K$ of nontrivial absolute values $| \cdot
 |_v$ on $K$ along with positive integers $e_v$ such that the product
 formula
 \[ \prod_{v \in \bM_K} |z|_v^{e_v} = 1 \]
 holds for all nonzero $z \in K$.  

 When $K$ is a number field, these are
 simply the usual archimedean and $p$-adic absolute values, suitably
 normalized.  When $K$ is a function
 field over a field $k$, we choose the absolute values from prime divisors on a
 variety $V$ over $k$ whose function field is a finite extension of
 $\bQ$ when we are in characteristic 0 and a finite extension of
 $\F_p$ when we are in characteristic $p$.   The set of $x \in K$ such
 that $|x|_v = 1$ for all $v \in \bM_K$ is called the {\em field of
   constants}.  

 By extending the $| \cdot |_v$ to $\Kbar$, we obtain a Weil height
 function on the projective space $\bP^n$ by defining
 \[ h_{\bP^n}(z_0: \dots: z_n) =  \frac{1}{m} \sum_{v \in \bM_K} \sum_{i=1}^m \log
   \max (|z_0^{[i]}|_v, \cdots, |z_n^{[i]}|_v) \]
 where $(z_0^{[i]}: \cdots:z_n^{[i]})$, $i=1, \dots, m$ is the set
 of conjugates of $(z_0: \cdots: z_n)$ in $\Kbar$ over $K$ (note that
 while this does depend on our choice of coordinates, a change of
 coordinates will only change the definition by a bounded constant --
 see \cite{DioGeo} or \cite{BG} for details).  

 When $\CL$ is an ample line bundle on $V$, we can associate a height
 function to $h$ to $\CL$ by letting $\iota: V \lra \bP^n$ be an
 embedding such that $\iota^* \cO_{\bP^n}(1) =\CL^{\bigotimes e}$ (such an $\iota$ and
 $e$ exist when $\CL$ is ample) and taking $h_\CL(z) = \frac{1}{e} h_{\bP^n}
 (\iota(z))$.  

If $\CL$ is an ample line bundle on $V$ with associated height function $h_\CL$  and
$\varphi^*\CL \cong \CL^{\bigotimes d}$, where $d >1$, we have
 \begin{equation}\label{Weil-bound}
   |h_\CL(\varphi(z)) - d h_\CL(z)| < C
   \end{equation}
   for all $V(\Kbar)$. We can attach a
   canonical height to $\varphi$ as in \eqref{can} (see \cite{CS93}) by letting
   $h_\varphi(z) = \lim_{n \to \infty} \frac{h_\CL(\varphi^n(z))}{d^n}$.  

Note that $h_\varphi(\varphi(z)) = d h_\CL(z)$ by construction, so if $z
\in \Prep(\varphi)$, then clearly $h_\varphi(z) = 0$. 
   
When $K$ is a number field or a finitely generated function field of characteristic $p$
with a finite field of constants and $h$ is a height function associated
to an ample line bundle $\CL$, we have the following (\cite{Northcott},
\cite[Section 1.2]{BakerFinite}). 

\begin{thm}\label{Northcott} (Northcott)
  Let $K$ be a number field or finitely generated function field in
  characteristic $p$.  Let $h_\CL$, $\varphi$, and $h_\varphi$ be as
  above; when $K$ is a function field, assume that its field of
  constants is finite.  For any constants $A$ and $B$ there are at
  most finitely many $z \in V(\Kbar)$ such that $h_\CL(z) \leq A$ and
  $[K(z): K] \leq B$.  Since $|h_\varphi - h_\CL|$ is bounded, this
  means that $h_\varphi(z) = 0$ if and only if $z \in \Prep(\varphi)$.
\end{thm}

Northcott's theorem does not hold over function fields of
characteristic 0 for Weil heights.  However, Moriwaki \cite{Mori1,
  Mori2} has used metrics on line bundles and Arakelov intersection
theory to associate a height function $\fh_\CL$ to an ample line bundle
$\CL$ such that a form of Northcott's theorem does hold.  As with Weil
heights, if $\varphi^*\CL \cong \CL^{\bigotimes d}$ for an ample line
bundle $\CL$, then by construction (see \cite[Section 2.4]{YZ13}) there
is a constant $C$ such that
\begin{equation}\label{Mori-bound}
 |\fh_\CL(\varphi(z)) - d \fh_\CL(z)| < C
  \end{equation}
for all $z \in V(\Kbar)$, where $\fh_\CL$ is a Moriwaki height associated
to $\CL$. We may form a canonical height $\fh_\varphi$ for $\fh_{\CL}$ by
taking the limit
\[ \fh_\varphi(z) = \lim_{n \to \infty}
  \frac{\fh_\CL(\varphi^n(z))}{d^n} .\]
We have the following (see \cite{Mori1, Mori2}).

\begin{thm}\label{Moriwaki} (Moriwaki)
  For any constants $A$ and $B$ there are at most finitely many
  $z \in V(\Kbar)$ such that $\fh_\CL(z) \leq A$ and
  $[K(z): K] \leq B$.  Since $|\fh_\varphi - \fh_\CL|$ is bounded,
  this means that $\fh_\varphi(z) = 0$ if and only if
  $z \in \Prep(\varphi)$.
\end{thm}

Now we are ready to prove Theorem \ref{common}.  We use Weil heights
to treat the case where the field $K$ is a number field or function
field of characteristic $p$.  For function fields of characteristic 0,
we use Moriwaki heights.  Note that we could treat the case of
rational functions over function fields of characteristic 0 using Weil
heights rather than Moriwaki heights, since Baker \cite{BakerFinite}
has proved a dynamical form of Northcott's theorem for Weil heights in
the case of rational functions, assuming a non-isotriviality
condition.  This may be possible in higher dimensions, too, as there
are more general dynamical Northcott-type results for non-isotrivial maps due to
Chatzidakis and Hrushovski \cite{CH1, CH2} (see also
\cite{GV}), but non-isotriviality
conditions there are a good deal more complicated.
   
 \begin{proof}[Proof of Theorem \ref{common}]

   If $K$ is a number field or a function field of characteristic $p$
   endowed with absolute values having a finite field of constants, we
   let $\tau = h_\CL$ and apply Proposition \ref{free} with
   $\tau_f = h_f$ and $\tau_g = h_g$.  Note that we meet the
   conditions of the proposition since \eqref{Weil-bound} holds.
   Thus, if $h_f \not= h_g$, then there is a $j$ such that
   $\langle f^j, g^j \rangle$ is a free semigroup on two generators.
   By Theorem \ref{Northcott}, if $\Prep(f) \not= \Prep(g)$, then there
   is a $z \in V(\Kbar)$ such that exactly one of $h_f(z)$ and
   $h_g(z)$ is zero, which means that $h_f$ cannot equal $h_g$.

   If $K$ is a function field, we argue the same way, letting
   $\tau = \fh_\CL$ with $\tau_f = \fh_f$ and $\tau_g = \fh_g$.  The
   conditions of Proposition \ref{free} hold by \eqref{Mori-bound} and
   the existence of a $z \in V(\Kbar)$ that is preperiodic for exactly
   one of $f$ and $g$ implies the existence of a $z \in V(\Kbar)$ such
   that $\fh_f(z) \not= \fh_g(z)$ (by Theorem \ref{Moriwaki}), which
   in turn implies that there is a $j$ such that
   $\langle f^j, g^j \rangle$ is a free semigroup on two generators by Proposition \ref{free}.

\end{proof}

The following corollary answers a conjecture posed by Cabrera and
Makienko \cite{CM}.  

\begin{cor}\label{same-measure}
  Let $f,g \in \bC(x)$ both have degree greater than 1.  Let $d \mu_f$
  and $d \mu_g$ be the measures of maximal entropy for $f$ and $g$,
  respectively.  If $d \mu_f \not= d \mu_g$ then here is a
  $j$ such that $\langle f^j, g^j \rangle$ is a free semigroup on two generators.
\end{cor}
\begin{proof}
 Theorem 1.5 of \cite{YZ13} (see also \cite{Carney}) states that if $\Prep(f) \cap \Prep(g)$ is
 infinite for
 $f, g \in \bC(x)$, then
 $\mu_f= \mu_g$.  
\end{proof}

\begin{remark}
  In fact, Theorem 1.5 of \cite{YZ13} implies that if
  $\Prep(f) = \Prep(g)$, then the canonical measures associated to $f$
  and $g$ are equal to each other in much more generality.  However,
  equality of these measures is a much weaker condition than equality
  of the set of preperiodic points.  For example, equality of measures
  of maximal entropy at a nonarchimedean place is much weaker than
  equality of the set of preperiodic points, since polarized morphisms
  having good reduction at a non-archimedean place $v$ will have the
  same canonical measure at $v$.  Even over $\bC$, one can have
  $\mu_f = \mu_g$ but $\Prep(f) \not= \Prep(g)$ for special rational
  functions $f$ and $g$ (let $f(X) = X^2$ and $g(X) = \omega X^2$
  where $|\omega| = 1$ but $\omega$ is not a root of unity, for example).  On the
  other hand, equality of measures of maximal entropy for non-special rational
  functions over $\bC$ has powerful consequences, due to work of Levin
  \cite{Levin1}; the results of Ye \cite{Ye} that we use in the next
  section rely on Levin's results.
\end{remark}

\subsection{Some counterexamples}\label{counter}
 
Propositions \ref{g1} and \ref{g2} provide a converse to Theorem \ref{common} for
rational functions in characteristic 0.  One might ask more generally,
if it is true that when $f,g: V \lra V$ are morphisms polarized by the
same line bundle for a variety $V$ in any dimension over any field, the
equality $\Prep(f) = \Prep(g)$ must imply that $\langle f, g \rangle$
cannot contain a nonabelian free semigroup.  It turns out
this is not true for polarized morphisms of varieties of dimension
greater than one in characteristic 0, as Example \ref{AV} shows.  In
characteristic $p$, it is not even true for polynomials, as Example
\ref{p-counter} shows.

\begin{example}\label{AV}
  Let $A$ be an abelian variety defined over a number field such that
  $\mathrm{End}^0(A) \otimes_\mathbb{Q} \mathbb{R}$ is the Hamiltonian
  quaternion algebra $\mathbb{H}$.  (That such abelian varieties exist
  is well-known, see Theorem B.33 of \cite{Lew99}, for example.) By
  Theorem 2 of \cite{GMS99}, $1+2i$ and $1+2j$ generate a free
  multiplicative subgroup of $\mathbb{H}$ of order $2$. We also have
  $\phi^\dagger \phi =[1-2i][1+2i]= [5]$ and
  $\psi^\dagger \psi = [1-2j][1+2j] = [5]$. Therefore $\phi$ and $\psi$
  are both polarized by the theta divisor $\Theta$ on $A$ by
  Proposition 3.1 of \cite{Paz13}.  Since $\phi$ and $\psi$ both
  commute with $[m]$ for any $m$, we must have
  $\Prep(\phi) = \Prep(\psi) = A_{\text {tors}}$.  
\end{example}

\begin{example}\label{p-counter}
  Let $K = \F_p$ and let $d, e > 1$ be integers such that
  $p \nmid de$.  Then if $f(x) = x^d$ and $g \in \F_p[t]$ is any
  polynomial of degree $e$ that is not a monomial, the semigroup
  $\langle f, g \rangle$ is a free semigroup on two generators by \cite[Lemma 3.1]{JZ}.  Note
  that $\Prep(f) = \Prep(g) = {\overline \F_p}$ since $f$ and $g$ are
  both defined over $\F_p$.
\end{example}

\section{Proofs of Theorems \ref{rational} and \ref{linear}}\label{linear-sec}

We will now prove Theorems \ref{rational} and \ref{linear}.  We will
do so by proving results on the growth of finitely generated semigroups
$\cS$ such that $\Prep(f) = \Prep(g)$ for all $f, g \in \cS$ with
degrees greater than 1; these are Propositions \ref{g1} and \ref{g2}.  Combining with Theorem \ref{common} then
gives Theorems \ref{rational} and \ref{linear}.   We begin with a
lemma about semigroups of maps that contain constant maps.

\begin{lem}\label{constant}
  Suppose $S$ is a finite set of maps from a set $\cU$ to itself and
  that $f$ is a map that sends all of $\cU$ to a single element of
  $\cU$. Let $S_1 = S \cup \{ f \}$.  If $S^{\leq n}$ is the
  set of distinct maps represented by words of length at most $n$
  in $S$ and $S_1^{\leq n}$ is the set of distinct maps represented
  by words of length at most $n$ in $S_1$, then $|S_1^{\leq n}| \leq
    2 |S^{\leq n}|$.
\end{lem}
\begin{proof}
  It will suffice to show that the number of words in $S_1^{\leq n}$
  containing $f$ is bounded by $|S^{\leq n}|$.  Let
  $w \in S_1^{\leq n}$ contain $f$.  We write $w = w_1 f w_2$ where
  $w_1$ does not contain $f$ (note that $w_1$ may be empty).  Let
  $v$ be the element of $\cU$ such that $f(u) = v$ for all
  $u \in \cU$.  Then $w(u) = w_1(v)$ for all $u \in \cU$.  Since
  $w_1 \in S^{\leq n}$, there are at most $|S^{\leq n}|$ such 
  $w_1(v)$, and our proof is done.
\end{proof} 

\begin{remark}
  While Lemma \ref{constant} allows us to treat semigroups of rational
  functions containing constant maps, we cannot expect to obtain
  results in higher dimensions for semigroups containing morphisms
  that are neither constant nor finite, because of the examples in
  \cite{Ok2}.
 \end{remark}

Throughout this section, $\cS$ will denote a finitely generated semigroup
of rational functions in $\bC(x)$ and $\cS^+$ will denote the set of
elements in $\cS$ of degree greater than 1.   

We begin with a simple lemma.

\begin{lem}\label{linfin}
Let $f \in \bC(x)$ be a non-special rational function of degree greater than
1.  Then the set of $\sigma \in \bC(x)$ of degree 1 such that $\mu_f
= \mu_{\sigma f}$ is finite.
\end{lem}
\begin{proof}
  Let $\cJ$ denote the Julia set of $f$ and $\sigma f$ and let $\mu$
  denote their measure of maximal entropy.  Then $\sigma(\cJ) = \cJ$
  and $\mu(A) = \mu(\sigma(A)$) for any set $A$ in $\bP^1(\bC)$.  Thus,
  $\sigma$ is a symmetry on $\cJ$ in the sense of \cite[Definition
  1]{Levin1}.  Since $f$ is non-special the set of such symmetries is
  finite by \cite[Theorem 1]{Levin1}.
 \end{proof} 

We will use the following result due to Ye \cite{Ye}.  Since his argument
appears with slightly different notation as an implication in the
proof of \cite[Theorem 1.5]{Ye} (on page 393), rather than as a lemma
or theorem, we provide a proof using the same argument here.

\begin{lem}\label{HexiLem}
  Let $f, g \in \bC(x)$ be rational functions such that
  $\mu_f = \mu_g$ and $\deg f = \deg g > 1$.  Suppose that there is an
  $\alpha \not= \infty$ in 
  $\cJ_f = \cJ_g$ such that
  $f(\alpha) = g(\alpha) = \alpha$ and $f'(\alpha) = g'(\alpha)$.
  Then $f = g$.
\end{lem}

\begin{proof}
  We will show that $R:= f \circ g^{-1}$ is the identity map. Otherwise,
  since $R$ has multiplier equaling $1$ at $\alpha$, it
  determines attracting and repelling flowers near $\alpha$. Suppose that
  there is some point $x$ near $\alpha$ in the Julia set that is also in some
  attracting petal of the flowers determined by $R$ (see \cite[Section
  10]{Milnor-Complex}). Then there is some
  fundamental domain of $R$ for this petal, which contains some
  neighborhood of this point $x$. As the measure of maximal entropy
  $\mu$ is supported in the Julia set, the fundamental domain won’t
  have zero measure. Since $R$ acts on this petal like a
  transformation (in appropriate coordinates) and $R$ preserves the
  measure $\mu$ (since $\deg f = \deg g$), the $\mu$-measure of this
  petal cannot be finite.
\end{proof}

\begin{lem}\label{derivative}
Let $f, g \in \bC(x)$ be non-special rational functions of degree
greater than 1  such that $\Prep(f) = \Prep(g)$ and $\deg f = \deg g$.  Let
$\cJ$ denote $\cJ_f = \cJ_g$.
Suppose there is a periodic cycle $\{ x_1, \dots, x_r \}$
for $f$ and $g$ in
$\cJ \cap \bC$ such that $f(x_i) = g(x_i) = x_{i+1}$
for $i = 1, \dots, r-1$ and $f(x_r) = g(x_r) = x_1$.
Then we have the following:
\begin{enumerate}
\item $f'(x_1)/g'(x_1)$ is a root of unity.
 \item If $f'(x_1) = g'(x_1)$, then $f = g$.  
\end{enumerate}\end{lem}
\begin{proof}
  Let $\alpha_i=f'(x_i)$ and let $\beta_i=g'(x_i)$ for $i=1, \dots,
  r$.  Note that none of the $\alpha_i$ and $\beta_i$ are zero since
  the $x_i$ are in the Julia set for $f$ and $g$.  
Now let $h_1=fg^{r-1}$ and let $h_2=g^r$.  Then $h_1(x_2) = h_2(x_2) =
x_2$.  We have
\begin{equation}\label{prod1}
  h_1'(x_2) = \alpha_1 \prod_{i=2}^n \beta_i
  \end{equation}
  and
  \begin{equation}\label{prod2}
h_2'(x_2) = \beta_1 \prod_{i=2}^n \beta_i
\end{equation}
by the chain rule.  

Since $\Prep(h_1) =
\Prep(h_2) = \Prep(f) = \Prep(g)$ and $\cJ$ contains a point that is
periodic for both $h_1$ and $h_2$, it follows from \cite[Theorem
1.5]{Ye} that there is an $n$ such that $h_1^n = h_2^n$.  By
\eqref{prod1} and \eqref{prod2}, this means that
$(\alpha_1/\beta_1)^n = 1$, so $f'(x_1)/g'(x_1)$ is a root of unity,
as desired.  Furthermore, if $f'(x_1) = g'(x_1)$, then \eqref{prod1}
and \eqref{prod2} imply that $h_1'(x_2) =
h_2'(x_2)$, which means that $h_1 = h_2$, by Lemma \ref{HexiLem}.
Now, if $f g^{r-1} = g^r$, then $f =g$ by right cancellation.
 \end{proof}

 \begin{lem}\label{deg-bound}
   Let $\cS$ be a finitely generated semigroup of rational functions
   in $\bC(x)$ such that the elements of $\cS^+$ are non-special.
   Suppose that $\cS$ is not empty and that $\Prep(f) = \Prep(g)$ for
   all $f,g \in \cS^+$.  Then there is a constant $N$ such that for
   all $d \geq 1$, the number of elements of $\cS$ of degree $d$ is
   less than or equal to $N$.
  \end{lem}
  \begin{proof}
    There are finitely many elements in $\cS$ of degree 1 by Lemma
    \ref{linfin}, since there is an $f \in \cS$ of degree greater than
    1 such that $\Prep(\sigma f) = \Prep(f)$ for all $\sigma \in \cS$
    of degree 1.

    Now, let $\cJ$ be the Julia of the elements in $\cS$ having degree
    greater than 1.  Let $\cP$ be the set of preperiodic points of the
    elements $\cS^+$.  Then $f(\cP) = \cP$ for all $f \in \cS^+$.
    Choose a $\gamma \in \bC \cap \cP$.  Let $K$ be a finitely
    generated field over which $\gamma$ and every element of $\cS$ is
    defined.  Since $\Prep(f) \cap K$ is finite for each $f$ of degree
    greater than 1, it follows that the orbit $\cO$ of $\gamma$ under
    $\cS^+$ is finite.  After change of coordinates, we may assume
    that $\cO$ does not contain the point at infinity.

    Let $n$ be the number of roots unity in $K$.  Let $f \in \cS^+$
    have degree $d$.  We will show that there are at most $n$ elements
    $g \in \cS^+$ such that $\deg g = d$ and $g|_\cO = f|_\cO$.  Let
    $\{ x_1, \dots, x_r \}$ be a periodic cycle for $f$ in $\cO$ (note
    that there must be one since $\cO$ is finite) such that
    $f(x_i) = x_{i+1}$ for $i=1,\dots, r-1$ and $f(x_r) = x_1$.  Then
    for any $g \in \cS$ such that $\deg g = d$ and $g|_\cO = f|_\cO$,
    we may apply Lemma \ref{derivative} to conclude that
    $f'(x_1)/g'(x_1)$ is a root of unity in $K$.  Furthermore, given
    any $g_1, g_2 \in \cS$ of degree $d$ such that
    $g_1|_\cO = g_2|_\cO = f|_\cO$, we have $g_1 = g_2$ whenever
    $g_1'(x_1) = g_2'(x_1)$.  Thus, the number of $g \in \cS$ such
    that $\deg g = d$ and $g|_\cO = f|_\cO$ is bounded by the number
    of roots of unity in $K$, which is $n$.  Since $\cO$ is finite,
    there are $|\cO|^{|\cO|}$ maps from $\cO$ to itself, so there are
    at most $n |\cO|^{|\cO|}$ elements of $\cS$ of degree $d$ for any
    $d>1$.
    \end{proof}

\begin{lem}\label{degree}
  Suppose that $\cS$ is a semigroup of rational functions in $\bC(x)$
  such that every element of $\cS^+$ is non-special and has the same
  set of preperiodic points.  Then there is some $a\ge 2$ such that
  every element of $\cS$ has degree $a^n$ for some $n\ge 0$.
\end{lem}
\begin{proof}
By \cite[Theorem 3]{Levin1}, for any $f,g \in \cS^+$ there are $m$ and $n$ such that $f^m g^n
= f^{2m}$ so $(\deg f)^m = (\deg g)^n$.   Since the set $\cU$ of possible
degrees of elements in $\cS^+$ is a finitely generated subsemigroup of
the positive natural numbers under multiplication, it follows that
$\cU$ is contained in a semigroup generated by a single element $a$.   
\end{proof}

Now we are ready to state and prove Proposition \ref{g1}.

\begin{prop}\label{g1}
  Let $\cS$ be a finitely generated semigroup of rational functions in
  $\bC(x)$ such that $\cS^+$ contains a non-special rational function
  of degree greater than 1.  Suppose that $\mu_f = \mu_g$ for all
  $f, g \in \cS^+$.  Then $\cS$ has linear growth.
\end{prop}

\begin{proof}
  It suffices to prove this when every element of $\cS$ is
  nonconstant, by Lemma \ref{constant}. Now, since some element of
  $\cS^+$ is non-special, by hypothesis, we see that all elements of
  $\cS^+$ are non-special (see \cite{Levin1}, for example).  Let
  $f_1,\ldots ,f_s$ be generators for $\cS$.  Then for each $i$, we
  have that there is some $a\ge 2$ such that ${\rm deg}(f_i)=a^{m_i}$,
  by Lemma \ref{degree}; let $M=\max(m_1,\ldots ,m_s)$.  Now consider
  the set $S^{\le n}$ of elements in $\cS$ formed by taking a
  composition of length $\le n$ of elements from $f_1,\ldots
  ,f_s$. These elements all have degree in
  $\{1,a,a^2,\ldots ,a^{Mn}\}$, so there is a natural number $N$ such
  that $|S^{\le n}| \leq MNn$ for all $n$, by Lemma \ref{deg-bound}.
  Clearly, $|S^{\le n}| \ge n+1$ since the elements
  ${\rm id},f,f^2,\ldots ,f^n$ are pairwise distinct for $f \in \cS^+$
  and so we see that $\cS$ has linear growth as claimed.
\end{proof}

We are now ready to prove state and prove Proposition \ref{g2}.  

\begin{prop}\label{g2}
  Let $\cS$ be a finitely generated semigroup
  of rational functions in $\bC(x)$.  Suppose that for any $f, g \in
  \cS^+$, we have $\Prep(f) = \Prep(g)$.  Then $\cS$ has polynomially bounded
  growth.   
\end{prop}

\begin{proof}
  Again, we may assume that every element of $\cS$ is nonconstant, by
  Lemma \ref{constant}.
  Because of Proposition \ref{g1}, we need only treat the cases where
  every element of $\cS$ is linear or $\cS^+$ contains a special
  rational function.  We treat them case-by-case.

\noindent {\bf Case I.  Every element of $\cS$ is linear}.  In this case, the result is
contained in \cite[Theorem 1.5]{Ok1}.

\noindent {\bf Case II. Some element of $\cS^+$ is conjugate to a Chebychev polynomial.}

Let $f \in \cS^+$ be conjugate to a Chebychev polynomial.  Then there is a
$\sigma \in \bC(x)$ such that
$\sigma^{-1} f \sigma = T_d \in \Qbar(x)$, where $T_d$ is the
Chebychev polynomial of degree $d$.  If there is some $g \in \cS^+$
such that $\sigma^{-1} g \sigma \notin \Qbar(x)$, then clearly
$\Prep(f) \not= \Prep(g)$.  Hence we may assume that
$\sigma^{-1} \cS^+ \sigma \subseteq \Qbar(x)$. Furthermore we may
assume that $h_{\sigma^{-1} g \sigma} = h_{T_d}$ for all $g \in \cS^+$
by Theorem \ref{Northcott}.  Thus, we may apply \cite[Theorem
1]{KawaSil} to conclude that for all $g \in \cS^+$, we have some $d$
such that $\sigma^{-1} g \sigma = \pm T_d$.  Let $S^{\le n}$ be the words of
length $n$ in some finite set $S$ of generators for $\cS$.  Then the number
of possible degrees for elements of $S^{\le n}$ is bounded by $O(n^s)$
where $s = |S|$, so the number of
elements of $S^{\le n}$ of degree greater than 1 is bounded by
$O(n^s)$ since there are at most two elements in $\cS^+$ having the
same degree by the above. The number of elements of $S^{\le n}$ of
degree 1 is also bounded by a polynomial in $n$ by \cite[Theorem
1.5]{Ok1}.  Hence $\cS$ has polynomially bounded growth.

\noindent {\bf Case III. Some element of $\cS^+$ is conjugate to $X^m$.}
 
There is some $f \in \cS^+$ and some $\sigma \in \bC(x)$ such that
$\sigma^{-1} f \sigma = X^m$ for some integer $m$ with $|m| > 2$.
Then, as in the Chebychev case, either
$\sigma g \sigma^{-1} \in \Qbar(x)$ for every $g \in \cS^+$.  By
\cite[Theorem 2]{KawaSil}, we must also have
$\sigma g \sigma^{-1} = \xi x^n$.  Since all the elements of $\cS$ are
defined over $K$, there are finitely many roots of unity $\xi$ that
are possible.  Thus there is a bound on the number of elements of
$\cS^+$ of any given degree.  Let $S^{\le n}$ be the words
in $\cS$ of length $n$ on some finite set of generators.  As in the
Chebychev case, the number of words of in $S^{\leq n}$ of degree
greater than 1 is bounded by $O(n^s)$ where $s$ is the number of
generators and the number of elements of $S^{\le n}$ of degree 1 is
also bounded by a polynomial in $n$ by \cite[Theorem 1.5]{Ok1}.

\noindent {\bf Case IV. Some element of $\cS^+$ is Latt\`es.}

Let $K$ be a finitely generated field such that every element of $\cS$
is defined over $K$.  Recall that a Latt\`es map for an elliptic curve
$E$ is a map $f: \bP^1 \lra \bP^1$ of degree greater than 1 such that
there are nonconstant maps $\phi: E \lra E$ and $\pi: E \lra \bP^1$
with the property that $f \pi = \pi \phi$.  Suppose that some element
$f$ of $\cS$ is a Latt\`es map for an elliptic curve $E$.  Since
$\Prep(g) = \Prep(f)$ for every $g \in \cS^+$, it follows from
\cite[Theorem 27]{KawaSil} that every element of $\cS^+$ is a Latt\`es
map for $E$.  Hence, for each $f \in \cS^+$, there is a
$\pi_f: E \lra \bP^1$ and a $\phi_f: E \lra E$ such that
$f \pi_f = \pi_f \phi_f$.  By \cite[Theorem 3.1]{Milnor-Lattes}, the
degree of $\pi_f$ is either 2,3,4, or 6 (note that the possibilities
of 3, 4, and 6 can only arise when the endomorphism ring of $E$ is an
order in an imaginary quadratic field containing roots of unity other
than $\pm 1$).  Furthermore, by Theorem \cite[Theorem 30]{KawaSil}, if
$\deg f = \deg g$ and $\deg \pi_f = \deg \pi_g$, then
$g = \sigma_g f \sigma_g^{-1}$ for some $\sigma_g \in K(x)$ of degree
1.  After replacing $K$ by a suitable finite extension, we may assume
that there are at least three points in $\Prep(f) \cap K$.  Since
$\sigma$ must send $\Prep(f) \cap K$ to itself, there are finitely
many possible $\sigma_g$.

So now we have that there is a finite set $\Pi$ of maps
$\pi: E \lra E$ such that for any $g \in \cS^+$ we have
$g \pi_g = \pi_g \theta_g$ for some $\pi_g \in \Pi$ and some
$\theta_g: E \lra E$.  Let $S^{\le n}$ be the set of elements of $\cS$
represented by words of length $n$ in some set of generators for
$\cS$.  Since there are at most $|\Pi|$ possible choices for $\pi_g$
for $g \in \cS^+\cap S^{\le n}$, we will done if we can show that the
number of $\theta_g$ for $g \in \cS^+ \cap S^{\le n}$ is bounded by a
polynomial in $n$ (as before, the number of elements of degree 1 in
$S^{\le n}$ is bounded by a polynomial in $n$ by \cite[Theorem
1]{Ok1}).  Let $s$ be the number of distinct degrees of elements in
some generating set of $\cS$.  The number of possible degrees of
elements in $S^{\le n}$ is then bounded by $O(n^s)$.  Now let $M$ be
the maximal degree of an element of our generating set and let
$\cP = \{p_1, \dots, p_t\}$ be the set of primes dividing the degrees
of the elements of our generating set.  Then every element of
$S^{\le n} \cap \cS^+$ has degree bounded by $M^n$ and is divisible
only by elements of $\cP$.  Now, each $\theta_g$ can be written as
$m_g + t_g$ where $m_g \in \End(E)$ and $t_g$ is translation by
torsion point.  Since there are only finitely many torsion points of
$E$ in $K$, only finitely many translations arise, and we know that
the number of different possible degrees of $m_g$ for
$g \in \cS^+ \cap S^{\le n}$ is bounded by a polynomial in $n$, we
will be done if we can show that the number of $m_g$ having fixed
degree for $g \in \cS^+ \cap S^{\le n}$ is bounded by a polynomial in
$n$.  If $\End(E) = \bZ$, this is clear (that number is 2).  Otherwise
$\End(E)$ is an order in an imaginary field.  In the ring of integers
of an imaginary quadratic field the number of ideals having norm
$N = p_1^{e_1} \cdots p_t^{e_t}$ is bounded by
$(e_1 + 1) \cdots (e_t + 1)$ (this can be seen by noting that an ideal
having norm $p^\ell$ has at most $\ell+1$ possible factorizations --
see \cite[Chapter 17]{Hardy-Wright}, for example).  Since there at
most 6 units in an order in an imaginary quadratic field, this means
that the number of elements having norm $N$, where $N$ is a number
whose only prime factors are in $\{p_1, \dots, p_t\}$ is bounded by
$O((\log_2 N)^t)$.  For $N \leq M^n$, we see that $(\log_2 N)^t$ is
bounded by a polynomial in $n$; thus, the number of $m_g: E \lra E$
corresponding to a $g \in \cS^+ \cap S^{\le n}$ having fixed degree
is indeed bounded by a polynomial in $n$, and our proof is complete.
\end{proof}

\begin{remark}\label{cx}
  The proof of Proposition \ref{g2} shows that for a finitely
  generated field $K$ and any rational function $f \in K(x)$ that is not a
  Latt\`es map, there is a constant
  $N(f,K)$ such that for any $d$, the set of rational functions $g$ of
  degree $d$ such that $\Prep(g) = \Prep(f)$ has at most
  by $N(f,K)$ elements.  This is clearly not true for Latt\`es maps associated
  to elliptic curves with complex multiplication since the
  number of elements of an order in a quadratic number field having
  fixed norm can be arbitrarily large.
\end{remark}

Now we can easily prove Theorems \ref{linear} and \ref{rational}.

\begin{proof}[Proof of Theorem \ref{linear}]
By Theorem \ref{common}, the semigroup $\cS$ contains a free
subsemigroup on two elements unless $\Prep(f) = \Prep(g)$ for all $f,
g \in \cS^+$.  If $\Prep(f) = \Prep(g)$ for all $f, g \in \cS^+$, then
$\cS$ has linear growth by Proposition \ref{g1}.  
\end{proof}

\begin{proof}[Proof of Theorem \ref{rational}]
  As above, the semigroup $\cS$ contains a free subsemigroup on two
  elements unless $\Prep(f) = \Prep(g)$ for all $f, g \in \cS^+$, by
  Theorem \ref{common}.  If $\Prep(f) = \Prep(g)$ for all
  $f, g \in \cS^+$, then $\cS$ has polynomially bounded growth by
  Proposition \ref{g2}.  
\end{proof}

\begin{cor}\label{fromBD}
  Let $f, g \in \bC(X)$ be two rational functions, each having degree
  greater than 1.  Then the following are equivalent:
  \begin{enumerate}
    \item[(i)] $\Prep(f) \cap \Prep(g)$ is infinite;
    \item[(ii)] $\Prep(f) = \Prep(g)$;
      \item[(iii)] for any $\varphi_1, \varphi_2 \in \langle f, g
        \rangle$, we have $\Prep(\varphi_1) = \Prep(\varphi_2)$;
  \item [(iv)] $\langle f, g \rangle$ has polynomial growth;
    \item [(v)] $\langle f, g \rangle$ does not contain a nonabelian
      free semigroup;
      \item[(vi)] for any $\ell  > 0$, the semigroup $\langle f^\ell, g^\ell
        \rangle$ is not the free semigroup on two generators.
  \end{enumerate}
\end{cor}
\begin{proof}
  It is clear that (iii) implies (ii), that (iv) implies (v), and that
  (v) implies (vi).  Theorem 1.2 of \cite{BD11} (see also \cite{Mim13,
    YZ17, YZ13, Carney}) states that (i) and (ii) are equivalent.  Proposition
  \ref{g2} shows that (iii) implies (iv).  By Theorem \ref{common}, we
  have (vi) implies (iii).  Hence, we will be done if we can show that
  (ii) implies (iii).  Assume (ii) holds.  Let
  $\cU = \Prep(f) = \Prep(g)$.  Then $f(\cU) \subseteq \cU$ and
  $g(\cU) \subseteq \cU$. Let $\varphi \in \langle f, g \rangle$.
  Then we have $\varphi(\cU) \subseteq \cU$. Since $\cU$ contains at
  most finitely many points defined over any finitely generated field
  by Theorem \ref{Moriwaki}, it follows that for any $z \in \cU$, the
  orbit of $z$ is finite under $\varphi$, so
  $\cU \subseteq \Prep(\varphi)$.  Since $\Prep(f)$ and $\Prep(g)$ are
  infinite, it follows from Theorem 1.2 of \cite{BD11} that we have
  $\Prep(\varphi) = \Prep(f) = \Prep(g)$.  Thus,
  $\Prep(\varphi_1) = \Prep(\varphi_2) = \Prep(f) = \Prep(g)$ for any
  $\varphi_1, \varphi_2 \in \langle f, g \rangle$.
  \end{proof} 
\section{Proofs of Theorems \ref{poly1} and \ref{poly2}}\label{poly-sec}

We will prove a theorem that is slightly more general than Theorem
\ref{poly1}.  First a bit of notation.   Let $\varphi(\alpha) =
\alpha$ for $\varphi$ a nonconstant rational function in $\bC(x)$ and
$\alpha \in \bP^1(\bC)$.  We
let $e_\varphi(\alpha)  \geq 1$ denote the degree of the term of
lowest positive degree 
in the formal power series expansion of $\varphi$ centered at $\alpha$
(i.e. in $\varphi$ written locally as an element of $K[[(X-\alpha)]]$).  

\begin{thm}\label{poly1-gen}
Let $\cS$ be a finitely generated semigroup of rational functions in
$K(X)$.  Suppose that one of the following statements holds:
\begin{enumerate}
\item Every nonconstant element of $\cS$ has degree 1.
  \item There is an $\alpha \in \bP^1(\bC)$ fixed by every element of
    $\cS$ such that that $e_f(\alpha) >1$ for some nonconstant $f \in \cS$ and
    $\Char K \nmid e_g(\alpha)$ for every nonconstant element of $\cS$.  
  \end{enumerate}
Then either $\cS$ contains a nonabelian free semigroup or
$\cS$ has polynomially bounded growth.  
\end{thm}

We begin with some lemmas from Jiang and Zieve.

\begin{lem}\label{JZ1} (\cite[Lemma 2.1]{JZ})
  Let $f \in K[[X]]$ have a lowest degree term of degree $m > 1$ where
  $(\Char K) \nmid m$.  Then, there is a finite extension $K'$ of $K$
  and an $L \in K'[[X]]$ with lowest degree term of degree 1 such that $L^{-1}
  \circ f \circ L
  = X^m$.  
  
\end{lem}

\begin{lem}\label{JZ2}(\cite[Lemma 3.1]{JZ})
Let $K$ be a field, let $m, n > 1$ be integers not divisible by
$\Char K$, and let $g \in K[[X]]$ have lowest-degree term of degree $n$. If $g$
has at least two terms then $\langle X^m, g \rangle$ is a free semigroup on two generators.  
\end{lem}

\begin{lem}\label{JZ3}(\cite[Lemma 3.2]{JZ})
  Let $K$ be a field, let $m, n > 1$ be integers, and let $\alpha \in K^*$ have
infinite order. Then $\langle X^m, \alpha X^n\rangle$ is a free semigroup on two generators.  
\end{lem}

Note that Lemma \ref{JZ3} is not true when $n=1$; for example if $f(X)
= X^2$ and $g(X) = 2x$, then $g^2\circ f = f \circ g$, but we still have the
following.

\begin{lem}\label{4}
  Let $K$ be a field and let $m$ be a positive integer that is not
  divisible by $\Char K$.  Then for any
  $\alpha \in K^*$ of infinite order, the semigroup $\langle X^m \circ
  \alpha X, \alpha X \circ X^m\rangle$ is a free semigroup on two generators.
  \end{lem}
  \begin{proof}
    We have $X^m \circ \alpha X = \alpha^m X^m$ and $\alpha X \circ X^m
    = \alpha X^m$
    Let $L = \alpha^{-1/(m-1)} X$.  Then $L^{-1} \circ  \alpha X^m \circ
    L = X^m$ and  $L^{-1} \circ  \alpha^m X^m \circ
    L = \alpha^{m-1} X^m$, and $\langle X^m, \alpha^{m-1} X^m \rangle$
    is a free semigroup on two generators by Lemma \ref{JZ3}.  
    \end{proof}

    We are now ready to prove Theorem \ref{poly1-gen}.

    \begin{proof}
      Again we may assume that every element of $\cS$ is nonconstant,
      because of Lemma \ref{constant}.  
      If all the elements of $\cS$ have degree 1, then this is
      \cite[Theorem 1]{Ok1}.  We choose coordinates so that
      $\alpha = 0$ and write each element of $\cS$ as a formal power
      series in $K[[X]]$.  After passing to a finite extension and
      conjugating by some $L \in K[[X]]$, there is some $f \in \cS$
      and some $L \in K[[X]]$ where the degree of the term of lowest
      degree of $L$ is 1 such that $L^{-1} f L = X^m$ where $m > 1$
      and $(\Char K) \nmid m$, by Lemma \ref{JZ1}.  If there is some
      $g \in K[[X]]$ with $e_g(\alpha) = n > 1$ such that $L^{-1} g L$
      is not equal to $\xi X^n$ for $\xi$ a root of unity, then
      $\langle f, g \rangle$ is a free semigroup on two generators by Lemmas \ref{JZ2} and
      \ref{JZ3}.  Similarly, if there is some $g \in K[[X]]$ with
      $e_g(\alpha) = 1$ such that $L^{-1} g L$ is not equal to
      $\alpha X$ for some $\alpha \in K^*$, then
      $\langle f, gf \rangle$ is a free semigroup on two generators by Lemma \ref{JZ2}; likewise, if there
      is some $g \in K[[X]]$ with $e_g(\alpha) = 1$ such that
      $L^{-1} g L$ is equal to $\alpha X$ where $\alpha \in K^*$ is
      not a root of unity, then $\langle fg, gf \rangle$ is a free semigroup on two generators by
      Lemma \ref{4}.

Thus, we are reduced to showing that a semigroup of the form $\langle
\xi_r X^{e_1}, \dots, \xi_r X^{e_r} \rangle$, where each $\xi_i$ is a
root of unity, has
polynomially bounded growth.  Let $S^{\leq t}$ be the set of words of length
$t$ or less in $\langle
\xi_r X^{e_1}, \dots, \xi_r X^{e_r} \rangle$.  Then the number of
possible degrees of elements of $S_t$ is bounded by $O(t^r)$.  Since
the group generated by  $\{ \xi_1, \dots, \xi_r \}$ is finite, this means that
the number of elements in $S_t$ is also bounded by $O(t^r)$ and we
are done.

\end{proof}

We are now ready to prove Theorem \ref{poly2} in a slightly more
general form.

\begin{thm}\label{poly2-gen}
Let $f, g \in K(x)$ be nonconstant rational functions with a fixed point $\alpha$ such that
$e_f(\alpha)$ and $e_g(\alpha)$ are both greater than 1 and not
divisible by
$\Char K$.  If $\Prep(f) \not= \Prep(g)$, then $\langle f, g \rangle$
is a free semigroup on two generators.
  
 \end{thm}
  
\begin{proof}
  By Lemma \ref{JZ1}, after taking a finite extension there is a
  $L \in K[[X]]$ such that $L^{-1} f L = X^m$ for $m >1$ and not
  divisible by
  $\Char K$.  By Lemmas \ref{JZ2} and \ref{JZ3}, the semigroup
  $\langle f, g \rangle$ must be free unless $L g L^{-1} = \xi X^n$
  for $\xi$ a root of unity.  As in the proof of Theorem
  \ref{poly1-gen}, the semigroup $\langle X^m, \xi X^n \rangle$ has
  polynomially bounded growth, so we see that $\langle f, g \rangle$
  is either free or has polynomially bounded growth.

 If $\Prep(f) \not= \Prep(g)$, then $\langle f, g \rangle$ contains a
free group on two elements by Theorem \ref{common} and thus does not
have polynomially bounded growth, so $\langle f, g \rangle$ must be free. 
\end {proof}
\section{Proof of Theorem \ref{av-thm}}\label{abelian-v}

We now prove Theorem \ref{av-thm}.  

\begin{proof}[Proof of Theorem \ref{av-thm}]
  By \cite[Theorem 2]{Iitaka}, any morphism $f: A \lra A$ can be
  written as $\psi +t_a$ where $\psi$ is a group endomorphism of $A$
  and $t_a$ is the translation-by-$a$ map for $a \in A$.  Thus, we can
  write $\cS$ as
  $\langle \psi_1 + t_{a_1}, \dots, \psi_m + t_{a_m} \rangle$ where
  each $\psi_i$ is a group endomorphism of $A$.  Let $K$ be a finitely
  generated field such that $A$ and all the $\psi_i$ are defined over $K$ and
  such that all the $a_i$ are in $A(K)$.  After passing to a finite
  extension, we may assume that $A(K)$ is Zariski dense in $A$.

  For any group abelian $B$, we let $T(B)$ denote the set of affine
  maps on $B$ -- that is, the set of maps from $G$ to itself that are
  compositions of group homomorphisms and translations.  Since $A(K)$
  is Zariski dense in $A$, each element of $\cS$ is uniquely
  determined by its action on $A(K)$, so there is a natural embedding
  $\iota: \cS \lra T(A(K))$.  By
  \cite{Neron}, the group $A(K)$ is finitely generated; we let $n$
  denote its free rank.  We may write
  $A(K) = A(K)_{\tors} \bigoplus G$, for some finitely generated free
  abelian group $G$ of rank $n$.  Then the projection map $\pi: A(K) \lra G$ gives rise
  to a natural map $\theta: T(A(K)) \lra T(G)$ such that
  \begin{equation}\label{inverse}
    |\theta^{-1}(h)|
    \leq (n |A(K)_{\tors}|)^{|A(K)_{\tors}|}
  \end{equation}
  for all $h \in T(G)$.  
  Let $\cS'$ be the image of
  $\cS$ under $\theta \circ \iota$.  Certainly, $\cS$ will contain a
  nonabelain free semigroup whenever $\cS'$ does; because of
  \eqref{inverse}, $\cS$ must also have polynomial growth whenever
  $\cS'$ does.  

  Since each element of $\cS$ is a finite morphism, each element of $\cS'$ must
  be finite-to-one.  Tensoring with $\bQ$, we thus see that $\cS'$ is
  isomorphic to a cancellative semigroup of $\Aff_n(\bQ)$, the group
  of invertible affine linear maps on an $n$-dimensional vector space over $\bQ$.
  Since $\Aff_n(\bQ)$ embeds into $\GL_{n+1}(\bQ)$, we may apply
  \cite[Theorem 1]{Ok1} to conclude that either $\cS'$ has
  polynomially bounded growth or $\cS'$ contains a nonabelian free semigroup, and our proof is complete.
\end{proof}

\begin{remark}
  The hypothesis that the morphisms $f:A \lra A$ are finite is
  necessary.  There are examples of semigroups of $3 \times 3$
  matrices with integer coefficients having intermediate growth in \cite{Ok2}.
  Thus, for any elliptic curve $E$, the endomorphism ring
  $\End(E \times E \times E)$ contains multiplicative semigroups of
  intermediate growth.
\end{remark}

\begin{cor}\label{all-curves}
  Let $C$ be an irreducible curve over $\bC$ and let $\cS$ be a finitely
  generated semigroup of morphisms from $C$ to itself.  Then
  either $\cS$ has polynomially bounded growth or $\cS$ contains a
  nonabelian free semigroup.
 \end{cor}
 \begin{proof}
   By Lemma \ref{constant}, we may assume that every element of $\cS$
   is nonconstant.  Any morphism $f: C \lra C$ extends to a morphism
   ${\tilde f}: C' \lra C'$ for $C'$ the normalization of the
   projective closure of $C$.  Hence we may assume that $C$ is
   projective and nonsingular.  If $C$ has genus greater than 1, then
   $\cS$ must be finite (see \cite[Ex. IV.5.2]{Har}, for example), so
   we may assume that $C$ is isomorphic either to $\bP^1$ or to an elliptic curve.
   Applying Theorems \ref{rational} and \ref{av-thm} then gives the
   desired conclusion, since every nonconstant map on an elliptic curve
   is finite.
  \end{proof}
\section{Further directions}\label{further-sec}

We close with some general questions. In all of these questions, $K$
will be a field of arbitrary characteristic. 

\begin{quest}\label{1}
Are there rational functions $f, g\in K(X)$ of degree greater than
one such that $\Prep(f) \not= \Prep(g)$ and $\langle f, g \rangle$ is
not a free semigroup on two generators?
\end{quest}

One might also ask for something weaker, namely that there is a $j$
depending only on $K$ such that $\langle f^j, g^j \rangle$ must be
free whenever $\Prep(f) \not= \Prep(g)$.  This might be thought of as
analogous to the uniform version of the Tits alternative proved by
Breuillard and Gelander \cite{EffectiveTits}.  Recent work of DeMarco,
Krieger, and Ye \cite{DKY} suggests it may be possible here to use
quantitative equidistribution techniques to get good
uniform bounds on $\lim \inf |h_f - h_g|$ where $h_f$ and $h_g$ are the
canonical heights associated to $f$ and $g$ as in Section \ref{height-section}, at
least in the case of number fields (see also \cite{FRL} and
\cite{PST}).  This might allow for more
precision in the conclusion of Proposition \ref{free}. 

\begin{quest}
Let $V$ be a projective variety, let $\CL$ be an ample line bundle on $V$, and
let $\cS$ be a finitely generated semigroup of morphisms $f$ that are
polarized by $\CL$.  Is it true that $\cS$ must either have polynomially
bounded growth or contain a nonabelian free semigroup?
 \end{quest}

 It might also be natural to ask for a version of the Tits alternative
 for semigroups of polarized maps that says something about the
 structure of the semigroups rather than the growth.  For example, one
 might ask if it is true that any finitely generated semigroup of
 morphisms polarized by the same line bundle must contain either a
 nilpotent subsemigroup of finite index or a free subsemigroup on two
 generators (there is a notion of nilpotence for semigroups due to
 Malcev \cite{Mal53}).  Grigorchuk \cite{Grig} has shown that finitely
 generated cancellative semigroups have polynomially bounded growth if and only if they have a group of left quotients with a
 nilpotent subgroup of finite index; this extends well-known work of
 Gromov \cite{Gro} from the group setting to the cancellative
 semigroup setting.  It follows immediately that any cancellative
 finitely generated semigroup of rational functions contains either a
 nilpotent subsemigroup of finite index or a free semigroup on two
 generators.  Thus, in the case of cancellative subgroups of rational
 functions over $\bC$, we do have a natural structural analog of the
 Tits alternative for linear groups.

 Semigroups of polarized morphisms are not cancellative in general,
 however, as noted in Section \ref{prelim}.  On the other hand, we can
 show that a finitely generated semigroup of polynomials in $\bC[z]$
 contains either a nilpotent subsemigroup of finite index or a
 nonabelian free semigroup.  We can also show that if a finitely
 generated semigroup $\cS$ of polarized morphisms contains a nilpotent
 subsemigroup of finite index then all the elements of $\cS$ have the
 same set of preperiodic points.  Finally, the proof of Theorem
 \ref{rational} can be modified to show that if $\cS$ is a finitely
 generated semigroup of rational functions over $\bC$ that does not
 contain a Latt\`es map, then either $\cS$ contains a nonabelian free
 semigroup or there is an $N$ such that the degree map is at most
 $N$-to-1 on $\cS^+$ (see Remark \ref{cx}).  It is not clear to us,
 though, what the right kinds of general structural questions are in
 the non-cancellative setting of polarized morphisms.

\providecommand{\bysame}{\leavevmode\hbox to3em{\hrulefill}\thinspace}
\providecommand{\MR}{\relax\ifhmode\unskip\space\fi MR }
\providecommand{\MRhref}[2]{%
  \href{http://www.ams.org/mathscinet-getitem?mr=#1}{#2}
}
\providecommand{\href}[2]{#2}

\end{document}